\documentclass[12pt,reqno]{article}

\usepackage[usenames]{color}

\usepackage[colorlinks=true,
linkcolor=webgreen,
filecolor=webbrown,
citecolor=webgreen]{hyperref}

\definecolor{webgreen}{rgb}{0,.5,0}
\definecolor{webbrown}{rgb}{.6,0,0}

\usepackage{fullpage}
\usepackage{float}

\usepackage{graphics,amsmath,amssymb}
\usepackage{graphicx}
\usepackage{amsthm}
\usepackage{amsfonts}
\usepackage{latexsym}
\usepackage{epsf}

\begin{document}

\theoremstyle{plain}
\newtheorem{theorem}{Theorem}
\newtheorem{corollary}[theorem]{Corollary}
\newtheorem{lemma}[theorem]{Lemma}
\newtheorem{proposition}[theorem]{Proposition}

\theoremstyle{definition}
\newtheorem{definition}[theorem]{Definition}
\newtheorem{example}[theorem]{Example}
\newtheorem{conjecture}[theorem]{Conjecture}

\theoremstyle{remark}
\newtheorem{remark}[theorem]{Remark}

\begin{center}
\vskip 1cm{\LARGE\bf Extension of Summation Formulas involving Stirling series}
\vskip 1cm
\large
Raphael Schumacher\\
\end{center}

\vskip .2 in

\begin{abstract}
\noindent This paper presents a family of rapidly convergent summation formulas for various finite sums of the form $\sum_{k=0}^{\lfloor x\rfloor}f(k)$, where $x$ is a positive real number.
\end{abstract}

\section{Introduction}
\label{sec:introduction}

In this paper we will use the Euler-Maclaurin summation formula \cite{3,5} to obtain rapidly convergent series expansions for finite sums involving Stirling series \cite{1}. Our key tool will be the so called Weniger transformation \cite{1}.\newline
\noindent For example, one of our summation formulas for the sum $\sum_{k=0}^{\lfloor x\rfloor}\sqrt{k}$, where $x\in\mathbb{R}^{+}$ is
\begin{displaymath}
\begin{split}
\sum_{k=0}^{\lfloor x\rfloor}\sqrt{k}&=\frac{2}{3}x^{\frac{3}{2}}-\frac{1}{4\pi}\zeta\left(\frac{3}{2}\right)-\sqrt{x}B_{1}(\{x\})+\sqrt{x}\sum_{k=1}^{\infty}(-1)^{k}\frac{\sum_{l=1}^{k}(-1)^{l}\frac{(2l-3)!!}{2^{l}(l+1)!}S_{k}^{(1)}(l)B_{l+1}(\{x\})}{(x+1)(x+2)\cdots(x+k)}\\
&=\frac{2}{3}x^{\frac{3}{2}}-\frac{1}{4\pi}\zeta\left(\frac{3}{2}\right)+\left(\frac{1}{2}-\{x\}\right)\sqrt{x}+\frac{\left(\frac{1}{4}\{x\}^{2}-\frac{1}{4}\{x\}+\frac{1}{24}\right)\sqrt{x}}{(x+1)}\\
&\quad+\frac{\left(\frac{1}{24}\{x\}^{3}+\frac{3}{16}\{x\}^{2}-\frac{11}{48}\{x\}+\frac{1}{24}\right)\sqrt{x}}{(x+1)(x+2)}+\frac{\left(\frac{1}{64}\{x\}^{4}+\frac{3}{32}\{x\}^{3}+\frac{21}{64}\{x\}^{2}-\frac{7}{16}\{x\}+\frac{53}{640}\right)\sqrt{x}}{(x+1)(x+2)(x+3)}\\
&\quad+\frac{\left(\frac{1}{128}\{x\}^{5}+\frac{19}{256}\{x\}^{4}+\frac{109}{384}\{x\}^{3}+\frac{29}{32}\{x\}^{2}-\frac{977}{768}\{x\}+\frac{79}{320}\right)\sqrt{x}}{(x+1)(x+2)(x+3)(x+4)}+\ldots,
\end{split}
\end{displaymath}
where the $B_{l}(x)$'s are the Bernoulli polynomials and $S_{k}^{(1)}(l)$ denotes the Stirling numbers of the first kind.

\vspace{0.4cm}

\noindent Most of the other formulas in this article have a similar shape.\newline
\noindent Setting in the above formula for $\sum_{k=0}^{\lfloor x \rfloor}\sqrt{k}$ the variable $x:=n\in\mathbb{N}$, we obtain
\begin{displaymath}
\begin{split}
\sum_{k=0}^{n}\sqrt{k}&=\frac{2}{3}n^{\frac{3}{2}}+\frac{1}{2}\sqrt{n}-\frac{1}{4\pi}\zeta\left(\frac{3}{2}\right)+\sqrt{n}\sum_{k=1}^{\infty}(-1)^{k}\frac{\sum_{l=1}^{k}(-1)^{l}\frac{(2l-3)!!}{2^{l}(l+1)!}B_{l+1}S_{k}^{(1)}(l)}{(n+1)(n+2)\cdots(n+k)}\\
&=\frac{2}{3}n^{\frac{3}{2}}+\frac{1}{2}\sqrt{n}-\frac{1}{4\pi}\zeta\left(\frac{3}{2}\right)+\frac{\sqrt{n}}{24(n+1)}+\frac{\sqrt{n}}{24(n+1)(n+2)}+\frac{53\sqrt{n}}{640(n+1)(n+2)(n+3)}\\
&\quad+\frac{79\sqrt{n}}{320(n+1)(n+2)(n+3)(n+4)}+\ldots,
\end{split}
\end{displaymath}
\noindent which is the corresponding formula given in our previous paper \cite{8}.

\section{Definitions}
\label{sec:Definitions}

\noindent As usual, we denote the floor of $x$ by $\lfloor x\rfloor$ and the fractional part of $x$ by $\{x\}$.

\begin{definition}(Pochhammer symbol)\cite{1}\newline
\noindent We define the \emph{Pochhammer symbol} (or rising factorial function) $(x)_{k}$ by 
\begin{displaymath}
(x)_{k}:=x(x+1)(x+2)(x+3)\cdots(x+k-1)=\frac{\Gamma(x+k)}{\Gamma(x)},
\end{displaymath}
where $\Gamma(x)$ is the gamma function defined by
\begin{displaymath}
\Gamma(x):=\int_{0}^{\infty}e^{-t}t^{x-1}dt.
\end{displaymath}
\end{definition}

\begin{definition}(Stirling numbers of the first kind)\cite{1}\newline
\noindent Let $k,l\in\mathbb{N}_{0}$ be two non-negative integers such that $k\geq l\geq0$. We define the \emph{Stirling numbers of the first kind} $S_{k}^{(1)}(l)$ as the connecting coefficients in the identity
\begin{displaymath}
(x)_{k}=(-1)^{k}\sum_{l=0}^{k}(-1)^{l}S_{k}^{(1)}(l)x^{l},
\end{displaymath}
where $(x)_{k}$ is the rising factorial function.
\end{definition}

\begin{definition}(Bernoulli numbers)\cite{2}\newline
\noindent We define the $k$-th \emph{Bernoulli number} $B_{k}$ as the $k$-th coefficient in the generating function relation
\begin{displaymath}
\frac{x}{e^{x}-1}=\sum_{k=0}^{\infty}\frac{B_{k}}{k!}x^{k}\;\;\forall x\in\mathbb{C}\;\text{with $|x|<2\pi$}.
\end{displaymath}
\end{definition}

\begin{definition}(Euler numbers)\cite{3}\newline
\noindent We define the sequence of \emph{Euler numbers} $\left\{E_{k}\right\}_{k=0}^{\infty}$ by the generating function identity
\begin{displaymath}
\frac{2e^{x}}{e^{2x}+1}=\sum_{k=0}^{\infty}\frac{E_{k}}{k!}x^{k}.
\end{displaymath}
\end{definition}

\begin{definition}(Bernoulli polynomials)\cite{2,3}\newline
\noindent We define for $n\in\mathbb{N}_{0}$ the $n$-th \emph{Bernoulli polynomial} $B_{n}(x)$ via the following exponential generating function as
\begin{displaymath}
\frac{te^{xt}}{e^{t}-1}=\sum_{n=0}^{\infty}\frac{B_{n}(x)}{n!}t^{n}\;\;\forall t\in\mathbb{C}\;\text{with $|t|<2\pi$}.
\end{displaymath}
\noindent Using this expression, we see that the value of the $n$-th Bernoulli polynomial $B_{n}(x)$ at the point $x=0$ is
\begin{displaymath}
B_{n}(0)=B_{n},
\end{displaymath}
which is the $n$-th Bernoulli number.
\end{definition}

\begin{definition}(Euler polynomials)\cite{3}\newline
\noindent We define for $n\in\mathbb{N}_{0}$ the $n$-th \emph{Euler polynomial} $E_{n}(x)$ via the following exponential generating function as
\begin{displaymath}
\frac{2e^{xt}}{e^{t}+1}=\sum_{n=0}^{\infty}\frac{E_{n}(x)}{n!}t^{n}.
\end{displaymath}
\noindent Moreover, we have that \cite{4}
\begin{displaymath}
\begin{split}
E_{n}(0)&=-2(2^{n+1}-1)\frac{B_{n+1}}{n+1}\;\;\forall n\in\mathbb{N}_{0}.
\end{split}
\end{displaymath}
\end{definition}

\section{Extended Summation Formulas involving\\
Stirling Series}
\label{sec:Extended Summation Formulas involving Stirling Series}

In this section we will prove our summation formulas for various finite sums of the form $\sum_{k=1}^{\lfloor x\rfloor}f(k)$. 
For this, we need the following

\begin{lemma}(Extended Euler-Maclaurin summation formula)\cite{3,5}\newline
\noindent Let $f$ be an analytic function. Then for all $x\in\mathbb{R}^{+}$, we have that
\begin{displaymath}
\begin{split}
\sum_{k=1}^{\lfloor x\rfloor}f(k)&=\int_{1}^{x}f(t)dt+\sum_{k=1}^{m}(-1)^{k}\frac{B_{k}(\{x\})}{k!}f^{(k-1)}(x)-\sum_{k=1}^{m}\frac{B_{k}}{k!}f^{(k-1)}(1)\\
&\quad+\frac{(-1)^{m+1}}{m!}\int_{1}^{x}B_{m}(\{t\})f^{(m)}(t)dt,
\end{split}
\end{displaymath}
where $B_{m}(x)$ is the $m$-th Bernoulli polynomial and $\{x\}$ denotes the fractional part of $x$. Therefore, for many functions $f$ we have the asymptotic expansion
\begin{displaymath}
\begin{split}
\sum_{k=1}^{\lfloor x\rfloor}f(k)\sim\int_{1}^{x}f(t)dt+C+\sum_{k=1}^{\infty}(-1)^{k}\frac{B_{k}(\{x\})}{k!}f^{(k-1)}(x)\;\;\text{as $x\rightarrow\infty$},
\end{split}
\end{displaymath}
for some constant $C\in\mathbb{C}$
\end{lemma}

\noindent and

\begin{lemma}(Extended Boole summation formula)\cite{3}\newline
\noindent Let $f$ be an analytic function. Then for all $x\in\mathbb{R}^{+}$, we have that
\begin{displaymath}
\begin{split}
\sum_{k=1}^{\lfloor x\rfloor}(-1)^{k+1}f(k)&=\frac{(-1)^{x-\{x\}}}{2}\sum_{k=0}^{m}(-1)^{k+1}\frac{E_{k}(\{x\})}{k!}f^{(k)}(x)-\sum_{k=0}^{m}\frac{(2^{k+1}-1)B_{k+1}}{(k+1)!}f^{(k)}(1)\\
&\quad+\frac{(-1)^{m}}{2m!}\int_{1}^{x}(-1)^{t-\{t\}}E_{m}(\{t\})f^{(m+1)}(t)dt,
\end{split}
\end{displaymath}
where $E_{m}(x)$ is the $m$-th Euler polynomial and $\{x\}$ denotes the fractional part of $x$.\newline
\noindent Therefore, for many functions $f$ we have the asymptotic expansion
\begin{displaymath}
\begin{split}
\sum_{k=1}^{\lfloor x\rfloor}(-1)^{k+1}f(k)\sim C+\frac{(-1)^{x-\{x\}}}{2}\sum_{k=0}^{\infty}(-1)^{k+1}\frac{E_{k}(\{x\})}{k!}f^{(k)}(x)\;\;\text{as $x\rightarrow\infty$},
\end{split}
\end{displaymath}
for some constant $C\in\mathbb{C}$.
\end{lemma}

\noindent From the above lemma we get by setting $x:=n\in\mathbb{N}$ the following

\begin{corollary}(Boole Summation Formula)\cite{3}\newline
\noindent Let $f$ be an analytic function. Then for all $n\in\mathbb{N}$, we have
\begin{displaymath}
\begin{split}
\sum_{k=1}^{n}(-1)^{k+1}f(k)&=(-1)^{n}\sum_{k=0}^{m}(-1)^{k}\frac{(2^{k+1}-1)B_{k+1}}{(k+1)!}f^{(k)}(n)-\sum_{k=0}^{m}\frac{(2^{k+1}-1)B_{k+1}}{(k+1)!}f^{(k)}(1)\\
&\quad+\frac{(-1)^{m}}{2m!}\int_{1}^{x}(-1)^{t-\{t\}}E_{m}(\{t\})f^{(m+1)}(t)dt,
\end{split}
\end{displaymath}
\noindent and the following asymptotic expansion
\begin{displaymath}
\begin{split}
\sum_{k=1}^{n}(-1)^{k+1}f(k)\sim C+(-1)^{n}\sum_{k=0}^{\infty}(-1)^{k}\frac{(2^{k+1}-1)B_{k+1}}{(k+1)!}f^{(k)}(n)\;\;\text{as $n\rightarrow\infty$},
\end{split}
\end{displaymath}
for some constant $C\in\mathbb{C}$.
\end{corollary}

\noindent As well as the next key result found by J. Weniger:

\begin{lemma}(Generalized Weniger transformation)\cite{1}\newline
\noindent For every inverse power series $\sum_{k=1}^{\infty}\frac{a_{k}(x)}{x^{k+1}}$, where $a_{k}(x)$ is any function in $x$, the following transformation formula holds
\begin{displaymath}
\begin{split}
\sum_{k=1}^{\infty}\frac{a_{k}(x)}{x^{k+1}}&=\sum_{k=1}^{\infty}\frac{(-1)^{k}}{(x)_{k+1}}\sum_{l=1}^{k}(-1)^{l}S_{k}^{(1)}(l)a_{l}(x)\\
&=\sum_{k=1}^{\infty}\frac{(-1)^{k}\sum_{l=1}^{k}(-1)^{l}S_{k}^{(1)}(l)a_{l}(x)}{x(x+1)(x+2)\cdots(x+k)}.
\end{split}
\end{displaymath}
\end{lemma}

\noindent Now, we prove as an example the following

\begin{theorem}(Extended summation formulas for the harmonic series)\newline
\noindent For every positive real number $x\in\mathbb{R}^{+}$, we have that
\begin{displaymath}
\begin{split}
\sum_{k=1}^{\lfloor x\rfloor}\frac{1}{k}&=\log(x)+\gamma-\frac{B_{1}(\{x\})}{x}+\sum_{k=1}^{\infty}(-1)^{k+1}\frac{\sum_{l=1}^{k}\frac{(-1)^{l}}{l+1}S_{k}^{(1)}(l)B_{l+1}(\{x\})}{x(x+1)(x+2)\cdots(x+k)}.
\end{split}
\end{displaymath}
\noindent We also have that
\begin{displaymath}
\begin{split}
\sum_{k=1}^{\lfloor x\rfloor}\frac{1}{k}&=\log(x)+\gamma+\sum_{k=1}^{\infty}(-1)^{k+1}\frac{\sum_{l=1}^{k}\frac{(-1)^{l}}{l}S_{k}^{(1)}(l)B_{l}(\{x\})}{(x+1)(x+2)\cdots(x+k)}.
\end{split}
\end{displaymath}
\end{theorem}

\begin{proof}
Applying the extended Euler-Maclaurin summation formula to the function $f(x):=\frac{1}{x}$, we get that
\begin{displaymath}
\sum_{k=1}^{\lfloor x\rfloor}\frac{1}{k}\sim\log(x)+\gamma-\sum_{k=1}^{\infty}\frac{B_{k}(\{x\})}{kx^{k}}.
\end{displaymath}
Applying now the generalized Weniger transformation to the equivalent series
\begin{displaymath}
\begin{split}
\sum_{k=1}^{\lfloor x\rfloor}\frac{1}{k}&\sim\log(x)+\gamma-\frac{B_{1}(\{x\})}{x}-\sum_{k=2}^{\infty}\frac{B_{k}(\{x\})}{kx^{k}}\\
&\sim\log(x)+\gamma-\frac{B_{1}(\{x\})}{x}-\sum_{k=1}^{\infty}\frac{B_{k+1}(\{x\})}{(k+1)x^{k+1}},
\end{split}
\end{displaymath}
we get the first claimed formula.
To obtain the second expression, we apply the generalized Weniger transformation to the identity
\begin{displaymath}
\sum_{k=1}^{\lfloor x\rfloor}\frac{1}{k}\sim\log(x)+\gamma-x\sum_{k=1}^{\infty}\frac{B_{k}(\{x\})}{kx^{k+1}}.
\end{displaymath}
\end{proof}

\noindent At this point, we want to give an overview on summation formulas obtained with this method:

\begin{itemize}
\item[1.)]{{\bf Extended summation formulas for the harmonic series:}\newline
\noindent For every positive real number $x\in\mathbb{R}^{+}$, we have that
\begin{displaymath}
\begin{split}
\sum_{k=1}^{\lfloor x\rfloor}\frac{1}{k}&=\log(x)+\gamma-\frac{B_{1}(\{x\})}{x}+\sum_{k=1}^{\infty}(-1)^{k+1}\frac{\sum_{l=1}^{k}\frac{(-1)^{l}}{l+1}S_{k}^{(1)}(l)B_{l+1}(\{x\})}{x(x+1)(x+2)\cdots(x+k)}
\end{split}
\end{displaymath}
and
\begin{displaymath}
\begin{split}
\sum_{k=1}^{\lfloor x\rfloor}\frac{1}{k}&=\log(x)+\gamma+\sum_{k=1}^{\infty}(-1)^{k+1}\frac{\sum_{l=1}^{k}\frac{(-1)^{l}}{l}S_{k}^{(1)}(l)B_{l}(\{x\})}{(x+1)(x+2)\cdots(x+k)}.
\end{split}
\end{displaymath}}

\item[2.)]{{\bf Extended summation formulas for the partial sums of $\zeta(2)$:}\newline
\noindent For every positive real number $x\in\mathbb{R}^{+}$, we have that
\begin{displaymath}
\begin{split}
\sum_{k=1}^{\lfloor x\rfloor}\frac{1}{k^{2}}&=\zeta(2)-\frac{1}{x}-\frac{B_{1}(\{x\})}{x^{2}}+\frac{1}{x}\sum_{k=1}^{\infty}(-1)^{k+1}\frac{\sum_{l=1}^{k}(-1)^{l}S_{k}^{(1)}(l)B_{l+1}(\{x\})}{x(x+1)(x+2)\cdots(x+k)}
\end{split}
\end{displaymath}
and
\begin{displaymath}
\begin{split}
\sum_{k=1}^{\lfloor x\rfloor}\frac{1}{k^{2}}&=\zeta(2)-\frac{1}{x}+\sum_{k=1}^{\infty}(-1)^{k+1}\frac{\sum_{l=1}^{k}(-1)^{l}S_{k}^{(1)}(l)B_{l}(\{x\})}{x(x+1)(x+2)\cdots(x+k)}.
\end{split}
\end{displaymath}}

\item[3.)]{{\bf Extended summation formulas for the partial sums of $\zeta(3)$:}\newline
\noindent For every positive real number $x\in\mathbb{R}^{+}$, we have that
\begin{displaymath}
\begin{split}
\sum_{k=1}^{\lfloor x\rfloor}\frac{1}{k^{3}}&=\zeta(3)-\frac{1}{2x^{2}}-\frac{B_{1}(\{x\})}{x^{3}}+\frac{1}{2x^{2}}\sum_{k=1}^{\infty}(-1)^{k+1}\frac{\sum_{l=1}^{k}(-1)^{l}(l+2)S_{k}^{(1)}(l)B_{l+1}(\{x\})}{x(x+1)(x+2)\cdots(x+k)}
\end{split}
\end{displaymath}
and
\begin{displaymath}
\begin{split}
\sum_{k=1}^{\lfloor x\rfloor}\frac{1}{k^{3}}&=\zeta(3)-\frac{1}{2x^{2}}+\frac{1}{2}\sum_{k=1}^{\infty}(-1)^{k+1}\frac{\sum_{l=1}^{k}(-1)^{l}(l+1)S_{k}^{(1)}(l)B_{l}(\{x\})}{x^{2}(x+1)(x+2)\cdots(x+k)}.
\end{split}
\end{displaymath}}

\item[4.)]{{\bf Extended summation formulas for the sum of the square roots:}\newline
\noindent For every positive real number $x\in\mathbb{R}^{+}$, we have that
\begin{displaymath}
\begin{split}
\sum_{k=0}^{\lfloor x\rfloor}\sqrt{k}&=\frac{2}{3}x^{\frac{3}{2}}-\frac{1}{4\pi}\zeta\left(\frac{3}{2}\right)-\sqrt{x}B_{1}(\{x\})+\sqrt{x}\sum_{k=1}^{\infty}(-1)^{k}\frac{\sum_{l=1}^{k}\frac{(-1)^{l}(2l-3)!!}{2^{l}(l+1)!}S_{k}^{(1)}(l)B_{l+1}(\{x\})}{(x+1)(x+2)\cdots(x+k)},
\end{split}
\end{displaymath}

\begin{displaymath}
\begin{split}
\sum_{k=0}^{\lfloor x\rfloor}\sqrt{k}&=\frac{2}{3}x^{\frac{3}{2}}-\frac{1}{4\pi}\zeta\left(\frac{3}{2}\right)-\sqrt{x}B_{1}(\{x\})+\frac{B_{2}(\{x\})}{4\sqrt{x}}+\sum_{k=1}^{\infty}(-1)^{k}\frac{\sum_{l=1}^{k}\frac{(-1)^{l}(2l-1)!!}{2^{l+1}(l+2)!}S_{k}^{(1)}(l)B_{l+2}(\{x\})}{\sqrt{x}(x+1)(x+2)\cdots(x+k)}
\end{split}
\end{displaymath}
and
\begin{displaymath}
\begin{split}
\sum_{k=0}^{\lfloor x\rfloor}\sqrt{k}&=\frac{2}{3}x^{\frac{3}{2}}-\frac{1}{4\pi}\zeta\left(\frac{3}{2}\right)+x\sqrt{x}\sum_{k=1}^{\infty}(-1)^{k}\frac{\sum_{l=1}^{k}\frac{(-1)^{l}(2l-5)!!}{2^{l-1}l!}S_{k}^{(1)}(l)B_{l}(\{x\})}{(x+1)(x+2)\cdots(x+k)}.
\end{split}
\end{displaymath}}

\item[5.)]{{\bf Extended summation formulas for the partial sums of $\zeta(-3/2)$:}\newline
\noindent For every positive real number $x\in\mathbb{R}^{+}$, we have that
\begin{displaymath}
\begin{split}
\sum_{k=0}^{\lfloor x\rfloor}k\sqrt{k}&=\frac{2}{5}x^{\frac{5}{2}}-\frac{3}{16\pi^{2}}\zeta\left(\frac{5}{2}\right)-x^{\frac{3}{2}}B_{1}(\{x\})+\frac{3}{2}x^{\frac{3}{2}}\sum_{k=1}^{\infty}(-1)^{k+1}\frac{\sum_{l=1}^{k}\frac{(-1)^{l}(2l-5)!!}{2^{l-1}(l+1)!}S_{k}^{(1)}(l)B_{l+1}(\{x\})}{(x+1)(x+2)\cdots(x+k)},
\end{split}
\end{displaymath}

\begin{displaymath}
\begin{split}
\sum_{k=0}^{\lfloor x\rfloor}k\sqrt{k}&=\frac{2}{5}x^{\frac{5}{2}}-\frac{3}{16\pi^{2}}\zeta\left(\frac{5}{2}\right)-x^{\frac{3}{2}}B_{1}(\{x\})+\frac{3}{4}\sqrt{x}B_{2}(\{x\})-\frac{B_{3}(\{x\})}{4\sqrt{x}}\\
&\quad+\frac{3}{2}\sum_{k=1}^{\infty}(-1)^{k+1}\frac{\sum_{l=1}^{k}\frac{(-1)^{l}(2l-1)!!}{2^{l+1}(l+3)!}S_{k}^{(1)}(l)B_{l+3}(\{x\})}{\sqrt{x}(x+1)(x+2)\cdots(x+k)}
\end{split}
\end{displaymath}
and
\begin{displaymath}
\begin{split}
\sum_{k=0}^{\lfloor x\rfloor}k\sqrt{k}&=\frac{2}{5}x^{\frac{5}{2}}-\frac{3}{16\pi^{2}}\zeta\left(\frac{5}{2}\right)+\frac{3}{2}x^{\frac{5}{2}}\sum_{k=1}^{\infty}(-1)^{k+1}\frac{\sum_{l=1}^{k}\frac{(-1)^{l}(2l-7)!!}{2^{l-2}l!}S_{k}^{(1)}(l)B_{l}(\{x\})}{(x+1)(x+2)\cdots(x+k)}.
\end{split}
\end{displaymath}
}

\item[6.)]{{\bf Extended summation formulas for the partial sums of $\zeta(-5/2)$:}\newline
\noindent For every positive real number $x\in\mathbb{R}^{+}$, we have that
\begin{displaymath}
\begin{split}
\sum_{k=0}^{\lfloor x\rfloor}k^{2}\sqrt{k}&=\frac{2}{7}x^{\frac{7}{2}}+\frac{15}{64\pi^{3}}\zeta\left(\frac{7}{2}\right)-x^{\frac{5}{2}}B_{1}(\{x\})+\frac{15}{4}x^{\frac{5}{2}}\sum_{k=1}^{\infty}(-1)^{k}\frac{\sum_{l=1}^{k}\frac{(-1)^{l}(2l-7)!!}{2^{l-2}(l+1)!}S_{k}^{(1)}(l)B_{l+1}(\{x\})}{(x+1)(x+2)\cdots(x+k)},
\end{split}
\end{displaymath}

\begin{displaymath}
\begin{split}
\sum_{k=0}^{\lfloor x\rfloor}k^{2}\sqrt{k}&=\frac{2}{7}x^{\frac{7}{2}}+\frac{15}{64\pi^{3}}\zeta\left(\frac{7}{2}\right)-x^{\frac{5}{2}}B_{1}(\{x\})+\frac{5}{4}x^{\frac{3}{2}}B_{2}(\{x\})-\frac{5}{8}\sqrt{x}B_{3}(\{x\})+\frac{5B_{4}(\{x\})}{64\sqrt{x}}\\
&\quad+\frac{15}{4}\sum_{k=1}^{\infty}(-1)^{k}\frac{\sum_{l=1}^{k}\frac{(-1)^{l}(2l-1)!!}{2^{l+1}(l+4)!}S_{k}^{(1)}(l)B_{l+4}(\{x\})}{\sqrt{x}(x+1)(x+2)\cdots(x+k)}
\end{split}
\end{displaymath}
and
\begin{displaymath}
\begin{split}
\sum_{k=0}^{\lfloor x\rfloor}k^{2}\sqrt{k}&=\frac{2}{7}x^{\frac{7}{2}}+\frac{15}{64\pi^{3}}\zeta\left(\frac{7}{2}\right)+\frac{15}{4}x^{\frac{7}{2}}\sum_{k=1}^{\infty}(-1)^{k}\frac{\sum_{l=1}^{k}\frac{(-1)^{l}(2l-9)!!}{2^{l-3}l!}S_{k}^{(1)}(l)B_{l}(\{x\})}{(x+1)(x+2)\cdots(x+k)}.
\end{split}
\end{displaymath}}

\item[7.)]{{\bf Extended summation formulas for the sum of the inverse square roots:}\newline
\noindent For every positive real number $x\in\mathbb{R}^{+}$, we have that
\begin{displaymath}
\begin{split}
\sum_{k=1}^{\lfloor x\rfloor}\frac{1}{\sqrt{k}}
&=2\sqrt{x}+\zeta\left(\frac{1}{2}\right)-\frac{B_{1}(\{x\})}{\sqrt{x}}+\sum_{k=1}^{\infty}(-1)^{k+1}\frac{\sum_{l=1}^{k}\frac{(-1)^{l}(2l-1)!!}{2^{l}(l+1)!}S_{k}^{(1)}(l)B_{l+1}(\{x\})}{\sqrt{x}(x+1)(x+2)\cdots(x+k)}
\end{split}
\end{displaymath}
and
\begin{displaymath}
\begin{split}
\sum_{k=1}^{\lfloor x\rfloor}\frac{1}{\sqrt{k}}
&=2\sqrt{x}+\zeta\left(\frac{1}{2}\right)+\sqrt{x}\sum_{k=1}^{\infty}(-1)^{k+1}\frac{\sum_{l=1}^{k}\frac{(-1)^{l}(2l-3)!!}{2^{l-1}l!}S_{k}^{(1)}(l)B_{l}(\{x\})}{(x+1)(x+2)\cdots(x+k)}.
\end{split}
\end{displaymath}}

\item[8.)]{{\bf Extended summation formulas for the partial sums of $\zeta(3/2)$:}\newline
\noindent For every positive real number $x\in\mathbb{R}^{+}$, we have that
\begin{displaymath}
\begin{split}
\sum_{k=1}^{\lfloor x\rfloor}\frac{1}{k\sqrt{k}}
&=\zeta\left(\frac{3}{2}\right)-\frac{2}{\sqrt{x}}-\frac{B_{1}(\{x\})}{x\sqrt{x}}+\frac{2}{\sqrt{x}}\sum_{k=1}^{\infty}(-1)^{k+1}\frac{\sum_{l=1}^{k}\frac{(-1)^{l}(2l+1)!!}{2^{l+1}(l+1)!}S_{k}^{(1)}(l)B_{l+1}(\{x\})}{x(x+1)(x+2)\cdots(x+k)}
\end{split}
\end{displaymath}
and
\begin{displaymath}
\begin{split}
\sum_{k=1}^{\lfloor x\rfloor}\frac{1}{k\sqrt{k}}
&=\zeta\left(\frac{3}{2}\right)-\frac{2}{\sqrt{x}}+2\sum_{k=1}^{\infty}(-1)^{k+1}\frac{\sum_{l=1}^{k}\frac{(-1)^{l}(2l-1)!!}{2^{l}l!}S_{k}^{(1)}(l)B_{l}(\{x\})}{\sqrt{x}(x+1)(x+2)\cdots(x+k)}.
\end{split}
\end{displaymath}}

\item[9.)]{{\bf Extended summation formulas for the partial sums of $\zeta(5/2)$:}\newline
\noindent For every positive real number $x\in\mathbb{R}^{+}$, we have that
\begin{displaymath}
\begin{split}
\sum_{k=1}^{\lfloor x\rfloor}\frac{1}{k^{2}\sqrt{k}}&=\zeta\left(\frac{5}{2}\right)-\frac{2}{3x^{\frac{3}{2}}}-\frac{B_{1}(\{x\})}{x^{2}\sqrt{x}}+\frac{4}{3x\sqrt{x}}\sum_{k=1}^{\infty}(-1)^{k+1}\frac{\sum_{l=1}^{k}\frac{(-1)^{l}(2l+3)!!}{2^{l+2}(l+1)!}S_{k}^{(1)}(l)B_{l+1}(\{x\})}{x(x+1)(x+2)\cdots(x+k)}
\end{split}
\end{displaymath}
and
\begin{displaymath}
\begin{split}
\sum_{k=1}^{\lfloor x\rfloor}\frac{1}{k^{2}\sqrt{k}}&=\zeta\left(\frac{5}{2}\right)-\frac{2}{3x^{\frac{3}{2}}}+\frac{4}{3\sqrt{x}}\sum_{k=1}^{\infty}(-1)^{k+1}\frac{\sum_{l=1}^{k}\frac{(-1)^{l}(2l+1)!!}{2^{l+1}l!}S_{k}^{(1)}(l)B_{l}(\{x\})}{x(x+1)(x+2)\cdots(x+k)}.
\end{split}
\end{displaymath}}

\item[10.)]{{\bf Extended Generalized Faulhaber Formulas:}\newline
\noindent For every positive real number $x\in\mathbb{R}^{+}$ and for every complex number $m\in\mathbb{C}\setminus\{-1\}$, we have that
\begin{displaymath}
\begin{split}
\sum_{k=1}^{\lfloor x\rfloor}k^{m}=\frac{1}{m+1}x^{m+1}+\zeta\left(-m\right)+\frac{x^{m+1}}{m+1}\sum_{k=1}^{\infty}(-1)^{k}\frac{\sum_{l=1}^{k}{m+1\choose l}S^{(1)}_{k}(l)B_{l}(\{x\})}{(x+1)(x+2)\cdots(x+k)}
\end{split}
\end{displaymath}
\noindent and
\begin{displaymath}
\begin{split}
\sum_{k=1}^{\lfloor x\rfloor}k^{m}=\frac{1}{m+1}x^{m+1}+\zeta\left(-m\right)-x^{m}B_{1}(\{x\})+\frac{x^{m}}{m+1}\sum_{k=1}^{\infty}(-1)^{k+1}\frac{\sum_{l=1}^{k}{m+1\choose l+1}S^{(1)}_{k}(l)B_{l+1}(\{x\})}{(x+1)(x+2)\cdots(x+k)}
\end{split}
\end{displaymath}
\noindent and
\begin{displaymath}
\begin{split}
\sum_{k=1}^{\lfloor x\rfloor}k^{m}&=\frac{1}{m+1}x^{m+1}+\zeta\left(-m\right)+\frac{1}{m+1}\sum_{k=1}^{\lfloor m+1\rfloor}(-1)^{k}{m+1\choose k}B_{k}(\{x\})x^{m-k+1}\\
&\quad+(-1)^{\lceil m+1\rceil}\frac{x^{m-\lceil m+1\rceil+2}}{m+1}\sum_{k=1}^{\infty}(-1)^{k+1}\frac{\sum_{l=1}^{k}{m+1\choose l+\lceil m+1\rceil-1}S^{(1)}_{k}(l)B_{l+\lceil m+1\rceil-1}(\{x\})}{(x+1)(x+2)\cdots(x+k)}.
\end{split}
\end{displaymath}
}

\item[11.)]{{\bf Generalized Faulhaber Formulas:}\newline
\noindent For every complex number $m\in\mathbb{C}\setminus\{-1\}$ and every natural number $n\in\mathbb{N}$, we have that
\begin{displaymath}
\begin{split}
\sum_{k=1}^{n}k^{m}=\frac{1}{m+1}n^{m+1}+\zeta\left(-m\right)+\frac{n^{m+1}}{m+1}\sum_{k=1}^{\infty}\frac{(-1)^{k}\sum_{l=1}^{k}{m+1\choose l}B_{l}S^{(1)}_{k}(l)}{(n+1)(n+2)\cdots(n+k)}
\end{split}
\end{displaymath}
\noindent and
\begin{displaymath}
\begin{split}
\sum_{k=1}^{n}k^{m}=\frac{1}{m+1}n^{m+1}+\zeta\left(-m\right)+\frac{1}{2}n^{m}+\frac{n^{m}}{m+1}\sum_{k=1}^{\infty}(-1)^{k+1}\frac{\sum_{l=1}^{k}{m+1\choose l+1}B_{l+1}S^{(1)}_{k}(l)}{(n+1)(n+2)\cdots(n+k)}
\end{split}
\end{displaymath}
\noindent and
\begin{displaymath}
\begin{split}
\sum_{k=1}^{n}k^{m}&=\frac{1}{m+1}n^{m+1}+\zeta\left(-m\right)+\frac{1}{m+1}\sum_{k=1}^{\lfloor m+1\rfloor}(-1)^{k}{m+1\choose k}B_{k}n^{m-k+1}\\
&\quad+(-1)^{\lceil m+1\rceil}\frac{n^{m-\lceil m+1\rceil+2}}{m+1}\sum_{k=1}^{\infty}(-1)^{k+1}\frac{\sum_{l=1}^{k}{m+1\choose l+\lceil m+1\rceil-1}B_{l+\lceil m+1\rceil-1}S^{(1)}_{k}(l)}{(n+1)(n+2)\cdots(n+k)}.
\end{split}
\end{displaymath}
}

\item[12.)]{{\bf Extended convergent versions of Stirling's formula:}\newline
\noindent For every positive real number $x\in\mathbb{R}^{+}$, we have that
\begin{displaymath}
\begin{split}
\sum_{k=1}^{\lfloor x\rfloor}\log(k)&=x\log(x)-x+\frac{1}{2}\log\left(2\pi\right)-\log(x)B_{1}(\{x\})+\sum_{k=1}^{\infty}(-1)^{k}\frac{\sum_{l=1}^{k}\frac{(-1)^{l}}{l(l+1)}S_{k}^{(1)}(l)B_{l+1}(\{x\})}{(x+1)(x+2)\cdots(x+k)},
\end{split}
\end{displaymath}

\begin{displaymath}
\begin{split}
\sum_{k=1}^{\lfloor x\rfloor}\log(k)&=x\log(x)-x+\frac{1}{2}\log\left(2\pi\right)-\log(x)B_{1}(\{x\})+\frac{B_{2}(\{x\})}{2x}\\
&\quad+\sum_{k=1}^{\infty}(-1)^{k}\frac{\sum_{l=1}^{k}\frac{(-1)^{l}}{(l+1)(l+2)}S_{k}^{(1)}(l)B_{l+2}(\{x\})}{x(x+1)(x+2)\cdots(x+k)}
\end{split}
\end{displaymath}
and that
\begin{displaymath}
\begin{split}
\sum_{k=1}^{\lfloor x\rfloor}\log(k)&=x\log(x)-x+\frac{1}{2}\log\left(2\pi\right)-\log(x)B_{1}(\{x\})+x\sum_{k=1}^{\infty}(-1)^{k}\frac{\sum_{l=2}^{k}\frac{(-1)^{l}}{l(l-1)}S_{k}^{(1)}(l)B_{l}(\{x\})
}{(x+1)(x+2)\cdots(x+k)}.
\end{split}
\end{displaymath}}

\item[13.)]{{\bf Extended first logarithmic summation formulas:}\newline
\noindent For every positive real number $x\in\mathbb{R}^{+}$, we have that
\begin{displaymath}
\begin{split}
\sum_{k=0}^{\lfloor x\rfloor}k\log(k)&=\frac{1}{2}x^{2}\log(x)-\frac{1}{4}x^{2}+\frac{1}{12}-\zeta^{'}(-1)-x\log(x)B_{1}(\{x\})+\frac{1}{2}\log(x)B_{2}(\{x\})\\
&\quad+\sum_{k=1}^{\infty}(-1)^{k+1}\frac{\sum_{l=1}^{k}\frac{(-1)^{l}}{l(l+1)(l+2)}S_{k}^{(1)}(l)B_{l+2}(\{x\})}{(x+1)(x+2)\cdots(x+k)}
\end{split}
\end{displaymath}
and
\begin{displaymath}
\begin{split}
\sum_{k=0}^{\lfloor x\rfloor}k\log(k)&=\frac{1}{2}x^{2}\log(x)-\frac{1}{4}x^{2}+\frac{1}{12}-\zeta^{'}(-1)-x\log(x)B_{1}(\{x\})+\frac{1}{2}\log(x)B_{2}(\{x\})-\frac{B_{3}(\{x\})}{6x}\\
&\quad+\sum_{k=1}^{\infty}(-1)^{k+1}\frac{\sum_{l=1}^{k}\frac{(-1)^{l}}{(l+1)(l+2)(l+3)}S_{k}^{(1)}(l)B_{l+3}(\{x\})}{x(x+1)(x+2)\cdots(x+k)}.
\end{split}
\end{displaymath}}

\item[14.)]{{\bf Extended second logarithmic summation formula:}\newline
\noindent For every positive real number $x\in\mathbb{R}^{+}$, we have that
\begin{displaymath}
\begin{split}
\sum_{k=1}^{\lfloor x\rfloor}\frac{\log(k)}{k}&=\frac{1}{2}\log(x)^{2}+\gamma_{1}-\frac{\log(x)}{x}B_{1}(\{x\})+\sum_{k=1}^{\infty}(-1)^{k}\frac{\sum_{l=1}^{k}\frac{(-1)^{l}}{(l+1)!}S_{l+1}^{(1)}(2)S_{k}^{(1)}(l)B_{l+1}(\{x\})}{x(x+1)(x+1)\cdots(x+k)}\\
&\quad+\log(x)\sum_{k=1}^{\infty}(-1)^{k}\frac{\sum_{l=1}^{k}\frac{1}{l+1}S_{k}^{(1)}(l)B_{l+1}(\{x\})}{x(x+1)(x+2)\cdots(x+k)}.
\end{split}
\end{displaymath}}

\item[15.)]{{\bf Extended third logarithmic summation formula:}\newline
\noindent For every positive real number $x\in\mathbb{R}^{+}$, we have that
\begin{displaymath}
\begin{split}
\sum_{k=1}^{\lfloor x\rfloor}\frac{\log(k)}{k^{2}}&=-\zeta^{'}(2)-\frac{\log(x)}{x}-\frac{1}{x}-\frac{\log(x)}{x^{2}}B_{1}(\{x\})\\
&\quad+\frac{1}{x}\sum_{k=1}^{\infty}(-1)^{k+1}\frac{\sum_{l=1}^{k}\frac{1}{l+1}\left(\sum_{m=0}^{l-1}\frac{m+1}{l-m}\right)S_{k}^{(1)}(l)B_{l+1}(\{x\})}{x(x+1)(x+2)\cdots(x+k)}\\
&\quad+\frac{\log(x)}{x}\sum_{k=1}^{\infty}\frac{(-1)^{k}\sum_{l=1}^{k}S_{k}^{(1)}(l)B_{l+1}(\{x\})}{x(x+1)(x+2)\cdots(x+k)}.
\end{split}
\end{displaymath}}

\item[16.)]{{\bf Extended fourth logarithmic summation formula:}\newline
\noindent For every positive real number $x\in\mathbb{R}^{+}$, we have that
\begin{displaymath}
\begin{split}
\sum_{k=1}^{\lfloor x\rfloor}\log(k)^2&=x\log(x)^{2}-2x\log(x)+2x+\frac{\gamma^{2}}{2}-\frac{\pi^{2}}{24}-\frac{\log(2)^{2}}{2}-\log(2)\log(\pi)-\frac{\log(\pi)^{2}}{2}+\gamma_{1}\\
&\quad-\log(x)^{2}B_{1}(\{x\})+\frac{\log(x)}{x}B_{2}(\{x\})+2\sum_{k=1}^{\infty}(-1)^{k}\frac{\sum_{l=1}^{k}\frac{(-1)^{l}}{(l+2)!}S_{l+1}^{(1)}(2)S_{k}^{(1)}(l)B_{l+2}(\{x\})}{x(x+1)(x+2)\cdots(x+k)}\\
&\quad+2\log(x)\sum_{k=1}^{\infty}(-1)^{k}\frac{\sum_{l=1}^{k}\frac{1}{(l+1)(l+2)}S_{k}^{(1)}(l)B_{l+2}(\{x\})}{x(x+1)(x+2)\cdots(x+k)}.
\end{split}
\end{displaymath}}

\item[17.)]{{\bf Extended summation formula for partial sums of the Gregory-Leibniz series:}\newline
\noindent For every positive real number $x\in\mathbb{R}^{+}$, we have that
\begin{displaymath}
\begin{split}
\sum_{k=0}^{\lfloor x\rfloor}\frac{(-1)^{k}}{2k+1}&=\frac{\pi}{4}+\frac{(-1)^{x-\{x\}}}{2}\sum_{k=0}^{\infty}\frac{(-1)^{k}\sum_{l=0}^{k}(-1)^{l}2^{l}S^{(1)}_{k}(l)E_{l}(\{x\})}{(2x+1)(2x+2)(2x+3)\cdots(2x+k+1)}.
\end{split}
\end{displaymath}

\noindent Setting $x:=n\in\mathbb{N}$, we get that
\begin{displaymath}
\begin{split}
\sum_{k=0}^{n}\frac{(-1)^{k}}{2k+1}&=\frac{\pi}{4}+(-1)^{n+1}\sum_{k=0}^{\infty}\frac{(-1)^{k}\sum_{l=0}^{k}\frac{(-1)^{l}}{l+1}2^{l}(2^{l+1}-1)S^{(1)}_{k}(l)B_{l+1}}{(2n+1)(2n+2)(2n+3)\cdots(2n+k+1)}.
\end{split}
\end{displaymath}}

\item[18.)]{{\bf Extended formula for the partial sums of the alternating harmonic series:}\newline
\noindent For every positive real number $x\in\mathbb{R}^{+}$, we have that
\begin{displaymath}
\begin{split}
\sum_{k=1}^{\lfloor x\rfloor}\frac{(-1)^{k+1}}{k}&=\log(2)+\frac{(-1)^{x-\{x\}}}{2}\sum_{k=0}^{\infty}(-1)^{k+1}\frac{\sum_{l=0}^{k}(-1)^{l}S^{(1)}_{k}(l)E_{l}(\{x\})}{x(x+1)(x+2)\cdots(x+k)}.
\end{split}
\end{displaymath}

\noindent Setting $x:=n\in\mathbb{N}$, we get
\begin{displaymath}
\begin{split}
\sum_{k=1}^{n}\frac{(-1)^{k+1}}{k}&=\log(2)+(-1)^{n}\sum_{k=0}^{\infty}(-1)^{k}\frac{\sum_{l=0}^{k}\frac{(-1)^{l}}{l+1}(2^{l+1}-1)S^{(1)}_{k}(l)B_{l+1}}{n(n+1)(n+2)\cdots(n+k)}.
\end{split}
\end{displaymath}}

\end{itemize}

\section{Other Extended Summation Formulas for Finite Sums}
\label{sec:Other Extended Summation Formulas for Finite Sums}

\noindent In this section, we denote by $\eta(s):=\sum_{k=1}^{\infty}\frac{(-1)^{k+1}}{k^{s}}$ the Dirichlet eta function.

\begin{itemize}
\item[1.)]{{\bf The extended alternating Faulhaber formula:}\begin{displaymath}
\begin{split}
\sum_{k=1}^{\lfloor x\rfloor}(-1)^{k+1}k^{m}&=\eta(-m)+\frac{(-1)^{x-\{x\}}}{2(m+1)}\sum_{k=0}^{m}(-1)^{k+1}(k+1){m+1\choose k+1}E_{k}(\{x\})x^{m-k}\;\;\forall m\in\mathbb{N}_{0}.
\end{split}
\end{displaymath}

\noindent Setting $x:=n\in\mathbb{N}$, we get
\begin{displaymath}
\begin{split}
\sum_{k=1}^{n}(-1)^{k+1}k^{m}&=\eta(-m)+\frac{(-1)^{n}}{m+1}\sum_{k=0}^{m}(-1)^{k}(2^{k+1}-1){m+1\choose k+1}B_{k+1}n^{m-k}\;\;\forall m\in\mathbb{N}_{0}.
\end{split}
\end{displaymath}}

\item[2.)]{{\bf The extended generalized alternating Faulhaber formula:}
\begin{displaymath}
\begin{split}
\sum_{k=1}^{\lfloor x\rfloor}(-1)^{k+1}k^{m}&=\eta(-m)+\frac{(-1)^{x-\{x\}}x^{m}}{2(m+1)}\sum_{k=0}^{\infty}(-1)^{k+1}\frac{\sum_{l=0}^{k}(l+1){m+1\choose l+1}S^{(1)}_{k}(l)E_{l}(\{x\})}{(x+1)(x+2)\cdots(x+k)}\;\;\forall m\in\mathbb{C}.
\end{split}
\end{displaymath}

\noindent Setting $x:=n\in\mathbb{N}$, we get
\begin{displaymath}
\begin{split}
\sum_{k=1}^{n}(-1)^{k+1}k^{m}&=\eta(-m)+\frac{(-1)^{n}n^{m}}{(m+1)}\sum_{k=0}^{\infty}(-1)^{k}\frac{\sum_{l=0}^{k}{m+1\choose l+1}(2^{l+1}-1)S^{(1)}_{k}(l)B_{l+1}}{(n+1)(n+2)\cdots(n+k)}\;\;\forall m\in\mathbb{C}.
\end{split}
\end{displaymath}}

\item[3.)]{{\bf The Geometric Summation Formula:}
\begin{displaymath}
\sum_{k=0}^{\lfloor x\rfloor}a^{k}=\frac{a^{x}}{\log(a)}+\frac{1}{1-a}+a^{x}\sum_{k=1}^{\infty}(-1)^{k}\frac{\sum_{l=1}^{k}\frac{\log(a)^{l-1}}{l!}S^{(1)}_{k}(l)B_{l}(\{x\})x^{l}
}{(x+1)(x+2)\cdots(x+k)}\;\;\forall a\neq1.
\end{displaymath}

\noindent Setting $x:=n\in\mathbb{N}$, we get
\begin{displaymath}
\sum_{k=0}^{n}a^{k}=\frac{a^{n}}{\log(a)}+\frac{1}{1-a}+a^{n}\sum_{k=1}^{\infty}(-1)^{k}\frac{\sum_{l=1}^{k}\frac{\log(a)^{l-1}}{l!}S^{(1)}_{k}(l)B_{l}n^{l}
}{(n+1)(n+2)\cdots(n+k)}\;\;\forall a\neq1.
\end{displaymath}}

\item[4.)]{{\bf The alternating Geometric Summation Formula:}
\begin{displaymath}
\sum_{k=0}^{\lfloor x\rfloor}(-1)^{k}a^{k}=\frac{1}{1+a}+\frac{(-1)^{x-\{x\}}}{2}a^{x}\sum_{k=0}^{\infty}(-1)^{k}\frac{\sum_{l=0}^{k}\frac{\log(a)^{l}}{l!}S^{(1)}_{k}(l)E_{l}(\{x\})x^{l}}{(x+1)(x+2)\cdots(x+k)}
\;\;\forall a\neq-1.
\end{displaymath}

\noindent Setting $x:=n\in\mathbb{N}$, we get
\begin{displaymath}
\sum_{k=0}^{n}(-1)^{k}a^{k}=\frac{1}{1+a}+(-1)^{n+1}a^{n}\sum_{k=0}^{\infty}(-1)^{k}\frac{\sum_{l=0}^{k}\frac{\log(a)^{l}}{(l+1)!}(2^{l+1}-1)S^{(1)}_{k}(l)B_{l+1}n^{l}}{(n+1)(n+2)\cdots(n+k)}
\;\;\forall a\neq-1.
\end{displaymath}}

\item[5.)]{{\bf The Euler-Maclaurin Geometric Summation Formula:}
\begin{displaymath}
\begin{split}
&\sum_{k=0}^{\lfloor x\rfloor}a^{k}=\frac{a^{x}}{\log(a)}+\frac{1}{1-a}+a^{x}\sum_{k=1}^{\infty}(-1)^{k}\frac{\log(a)^{k-1}}{k!}B_{k}(\{x\})\;\;\text{for}\;\;\frac{1}{e^{2\pi}}<a\neq1<e^{2\pi}.
\end{split}
\end{displaymath}}

\item[6.)]{{\bf The Euler-Maclaurin alternating Geometric Summation Formula:}
\begin{displaymath}
\begin{split}
\sum_{k=0}^{\lfloor x\rfloor}(-1)^{k}a^{k}&=\frac{1}{1+a}+(-1)^{x-\{x\}}\frac{a^{x}(a-1)}{a+1}\sum_{k=0}^{\infty}(-1)^{k}\frac{\log(a)^{k-1}}{k!}B_{k}(\{x\})\;\;\text{for}\;\;\frac{1}{e^{2\pi}}<a<e^{2\pi}.
\end{split}
\end{displaymath}}

\item[7.)]{{\bf The Exponential Geometric Summation Formula:}\newline
\begin{displaymath}
\begin{split}
\sum_{k=0}^{\lfloor x\rfloor}e^{k}&=e^{x}+\frac{1}{1-e}+e^{x}\sum_{k=1}^{\infty}(-1)^{k}\frac{B_{k}(\{x\})}{k!}\;\;\forall x\in\mathbb{R}^{+}.
\end{split}
\end{displaymath}}

\item[8.)]{{\bf The Self-Counting Summation Formula:}\newline
\noindent Let $\{a_{k}\}_{k=1}^{\infty}:=\{1,2,2,3,3,3,4,4,4,4,...\}$ be the self-counting sequence \cite{6,7} defined by
\begin{displaymath}
\begin{split}
a_{k}&:=\left\lfloor\frac{1}{2}+\sqrt{2k}\right\rfloor.
\end{split}
\end{displaymath}
\noindent Then, we have that
\begin{displaymath}
\begin{split}
\sum_{k=1}^{\lfloor x\rfloor}a_{k}&=
\frac{x\sqrt{8x+1}}{3}-\frac{5\sqrt{8x+1}}{24}-\frac{\sqrt{8x+1}}{2}B_{1}(\{x\})+B_{1}\left(\left\{\frac{\sqrt{8x+1}}{2}-\frac{1}{2}\right\}\right)B_{1}(\{x\})\\
&\quad+\frac{1}{2}B_{1}\left(\left\{\frac{\sqrt{8x+1}}{2}-\frac{1}{2}\right\}\right)-\frac{\sqrt{8x+1}}{4}B_{2}\left(\left\{\frac{\sqrt{8x+1}}{2}-\frac{1}{2}\right\}\right)\\
&\quad+\frac{1}{6}B_{3}\left(\left\{\frac{\sqrt{8x+1}}{2}-\frac{1}{2}\right\}\right).
\end{split}
\end{displaymath}}

\item[9.)]{{\bf Slowly convergent summation formula for the sum of the square roots:}
\begin{displaymath}
\begin{split}
\sum_{k=0}^{\lfloor x\rfloor}\sqrt{k}&=\frac{2}{3}x^{\frac{3}{2}}-\sqrt{x}B_{1}(\{x\})-\frac{1}{2\pi}\sum_{k=1}^{\infty}\frac{\text{FresnelS}\left(2\sqrt{k}\sqrt{x}\right)}{k^{\frac{3}{2}}}.
\end{split}
\end{displaymath}}

\item[10.)]{{\bf Slowly convergent summation formula for the harmonic series:}
\begin{displaymath}
\begin{split}
\sum_{k=1}^{\lfloor x\rfloor}\frac{1}{k}&=\log(x)+\gamma+2\sum_{k=1}^{\infty}\text{CosIntegral}(2\pi kx).
\end{split}
\end{displaymath}}
\end{itemize}

\section{Conclusion}
\label{sec:Conclusion}

\noindent This paper presents an overview on some rapidly convergent summation formulas obtained by  applying the Weniger transformation \cite{1}.

\bigskip
\hrule
\bigskip

\noindent 2010 {\it Mathematics Subject Classification}: Primary 65B15; Secondary 11B68.

\noindent\emph{Keywords:} Rapidly convergent summation formulas for finite sums of functions, extended Stirling summation formulas, extended Euler-Maclaurin summation formula, extended Boole summation formula, generalized Weniger transformation, Stirling numbers of the first kind, Bernoulli numbers, Bernoulli polynomials, Euler numbers, Euler polynomials.

\end{document}